\documentclass[11pt,oneside]{amsart}
\usepackage{amssymb, amsmath, amsthm,mathrsfs,calligra}
\usepackage{thmtools, thm-restate}
\usepackage[foot]{amsaddr}
\usepackage{hyperref}
\hypersetup{
	pdfauthor   = {Nils Bruin and Eugene Filatov},
	pdftitle    = {Twists of the Burkhardt quartic threefold},
	backref=true, pagebackref=true, hyperindex=true, colorlinks=true,
	breaklinks=true, urlcolor=blue, linkcolor=blue, citecolor=blue,
	bookmarks=true, bookmarksopen=true}
\usepackage[alphabetic,backrefs,lite]{amsrefs}
\usepackage[height=9in,width=6in,marginparwidth=0.8in]{geometry}
\usepackage{tikz-cd}
\usepackage{setspace}
\usepackage{mathtools}
\usepackage{marginnote}
\usepackage{multirow}
\usepackage{enumitem}

\newcommand{\sA}{\mathcal{A}}
\newcommand{\sM}{\mathcal{M}}

\renewcommand{\AA}{\mathbb{A}}
\newcommand{\CC}{\mathbb{C}}
\newcommand{\FF}{\mathbb{F}}
\newcommand{\PP}{\mathbb{P}}
\newcommand{\QQ}{\mathbb{Q}}
\newcommand{\RR}{\mathbb{R}}
\newcommand{\ZZ}{\mathbb{Z}}

\newcommand{\sep}{\mathrm{sep}}
\newcommand{\ksep}{k^\sep}
\newcommand{\Sigmabar}{\overline{\Sigma}}
\newcommand{\Gammabar}{\overline{\Gamma}}

\DeclareMathOperator{\He}{He}
\DeclareMathOperator{\Gal}{Gal}
\DeclareMathOperator{\Aut}{Aut}
\DeclareMathOperator{\Br}{Br}

\DeclareMathOperator{\Jac}{Jac}
\DeclareMathOperator{\Kum}{Kum}

\DeclareMathOperator{\Res}{Res}
\DeclareMathOperator{\Ob}{Ob}
\DeclareMathOperator{\GL}{GL}
\DeclareMathOperator{\SL}{SL}

\DeclareMathOperator{\PSp}{PSp}
\DeclareMathOperator{\PGSp}{PGSp}
\DeclareMathOperator{\PGL}{PGL}
\DeclareMathOperator{\Sp}{Sp}

\DeclareMathOperator{\Sym}{Sym}

\DeclareMathOperator{\HH}{H}

\newtheorem{theorem}{Theorem}[section]
\newtheorem{prop}[theorem]{Proposition}
\newtheorem{lemma}[theorem]{Lemma}

\theoremstyle{definition}

\newtheorem{remark}[theorem]{Remark}

\newtheorem{question}[theorem]{Question}

\begin{document}
\title{Twists of the Burkhardt Quartic Threefold}
\date{July 22, 2022}
\author{Nils Bruin}
\address[1]{Department of Mathematics, Simon Fraser University, Burnaby, BC, V5A 1S6, Canada}
\email{nbruin@sfu.ca}
\author{Eugene Filatov}
\address[2]{Department of Pure Mathematics, University of Toronto}
\email{eugene.filatov@utoronto.ca}

\thanks{The first author acknowledges the support of the Natural Sciences and Engineering Research Council of Canada (NSERC), funding reference number RGPIN-2018-04191.}

%14H10   	Families, moduli (algebraic)
%14K10   	Algebraic moduli, classification
%11G10   	Abelian varieties of dimension $> 1$ 
%11G18   	Arithmetic aspects of modular and Shimura varieties


\begin{abstract} We study twists of the Burkhardt quartic threefold over non-algebraically closed base fields of characteristic different from $2,3,5$. We show they all admit quartic models in projective four-space. We identify a Galois-cohomological obstruction that measures if a given twist is birational to a moduli space of abelian varieties.
This obstruction has implications for the rational points on these varieties. As a result, we see that all possible $3$-level structures can be realized by abelian surfaces, whereas Kummer $3$-level structures that group-theoretically may be admissible, may not be realizable over certain base fields. We give an example of a Burkhardt quartic over a bivariate function field whose desingularization has no rational points at all.
	
Our methods are based on the representation theory of $\Sp_4(\FF_3)$, Galois cohomology, and the classical algebraic geometry of the Burkhardt quartic. 
\end{abstract}
\maketitle
\section{Introduction and results}

The Burkhardt quartic threefold
\[B^{(1)}\colon y_0 (y_0^3 + y_1^3 + y_2^3 + y_3^3 + y_4^3 ) + 3y_1 y_2 y_3 y_4=0,\]
has received significant study both classically over $\CC$ (see \cites{Klein1888, Burkhardt1891, Maschke1889, Coble1917, Hunt96}) and more recently arithmetically. For instance, in \cites{BruinNasserden2018, CalChid2020} the rationality and non-rationality of certain twists of the Burkhardt quartic over $\QQ$ is established, while in \cites{CalChidRob2020, BoxCalGeePil2021} it is remarked that twists of the Burkhardt quartic that parametrize abelian surfaces, are unirational.

In this paper, we consider twists of $B^{(1)}$ over base fields $k$ of characteristic distinct from $2,3,5$. By a \emph{twist of $B^{(1)}$} we mean a variety $B$ over $k$ that, when base changed to a separable closure $\ksep$, is isomorphic to $B^{(1)}$. We refer to such a variety $B$ as \emph{a} Burkhardt quartic over $k$. This terminology suggests that $B$ can indeed be realized as a quartic threefold in $\PP^4$. This is true, but requires proof. It, and some other basic facts, follows quite directly from the representation theory of $\Sp_4(\FF_3)$. We collect these in the following theorem. See Section~\ref{S:proof_burkhardt_rep} for the proof.

\begin{restatable}{theorem}{thmburkhardtrep}
\label{T:burkhardt_rep}%
Let $k$ be a field of characteristic distinct from $2,3,5$ and let $B$ be a twist of $B^{(1)}$.
	\begin{enumerate}[label=(\alph*)]
		\item \label{rep:part1} $B$ admits a quartic model in $\PP^4$ with $45$ singularities.
		\item \label{rep:part2} $B$ comes equipped with a rational map $\pi\colon M\to B$ of generic degree $6$, where $M$ is a Brauer-Severi variety of dimension $3$.
		We write $\Ob(B)$ for the class of $M$ in $\Br(k)$.
		\item \label{rep:part3} $\Ob(B)\in \Br(k)$ is of period dividing $2$ and of index dividing $4$.
	\end{enumerate}
\end{restatable}

It is well-known that $B^{(1)}$ is a birational model of the moduli space of principally polarized abelian surfaces with full $3$-level structure. Outside the Hessian locus $\He(B)$ on $B$, a point $\alpha$ corresponds to the Jacobian $A_\alpha=\Jac(C_\alpha)$ of a genus $2$ curve $C_\alpha$, together with an isomorphism $\Sigma^{(1)}\to A_\alpha[3]$.
Here, $\Sigma^{(1)}=(\ZZ/3\ZZ)^2\times (\mu_3)^2$ is equipped with a natural pairing $\Sigma^{(1)}\times \Sigma^{(1)} \to \mu_3$ and the isomorphism is compatible with the Weil-pairing on $A_\alpha[3]$.

The intersection $B\cap \He(B)$ consists of a union of $40$ planes over $\ksep$, called $j$-planes.

The rational map $\pi\colon M\to B$, which is regular outside of $\He(B)$, has a moduli interpretation as well. It corresponds to marking an odd theta characteristic: a rational Weierstrass point on $C_\alpha$. The $M$ here is referred to as the \emph{Maschke $\PP^3$}.

For our purposes, it is more natural to think of $B^{(1)}$ as a moduli space of Kummer surfaces $K_\alpha=A_\alpha/\langle -1\rangle$. These come with a marked singularity (the image of the identity element of $A_\alpha$) as well as a Kummer $3$-level structure $\Sigmabar^{(1)}=\Sigma^{(1)}/\langle -1 \rangle$.

In general a Kummer surface $K$ over $k$, with one of the $16$ singular points marked, is a quotient of an abelian surface $A$ over $\ksep$. It does not fully determine $A$ over $k$: if $A$ admits a model over $k$, then any quadratic twist of $A$ has a Kummer surface isomorphic to $K$ as well. In fact, there may be no such abelian surface over $k$ at all. This is measured by $\Ob(K)\in\Br(k)$ and is represented by a conic $Q_K$. Equivalently, this obstruction arises from the moduli determining the curve $C_\alpha$. These moduli determine a curve of genus $0$ with a degree-$6$ locus marked, but only if that genus $0$ curve is actually a $\PP^1$ can one realize a double cover ramified over the marked locus.

The Kummer $3$-level structure already detects $\Ob(K)$. We establish that any Burkhardt quartic $B$ parametrizes Kummer surfaces with prescribed Kummer $3$-level structure, so it follows that the obstruction map is constant and hence is a function of $B$ itself. We collect results about it, and implications for the rational points on $B$, in the theorem below, that we prove in Section~\ref{S:proof_burkhardt_ob}.

\begin{restatable}{theorem}{thmburkhardtob}
\label{T:burkhardt_ob}%
Let $B$ be a Burkhardt quartic over a field $k$ of characteristic not $2,3,5$.
\begin{enumerate}[label=(\alph*)]
	\item\label{ob:part1} Then $B$ is naturally birational to the moduli space of Kummer surfaces with a Kummer $3$-level structure $\Sigmabar$.
	\item\label{ob:part2} If $\alpha \in B(k) \setminus \He(B)(k)$, then $\Ob(B)=\Ob(K_\alpha)$.
	\item\label{ob:part3} If $\He(B)\cap B$ contains a $j$-plane defined over $k$ then $\Ob(B)=1$.
	\item\label{ob:part4} If the Kummer $3$-level structure $\Sigmabar$ is a quotient of a full $3$-level structure $\Sigma$ over $k$, then $B(k)$ is Zariski-dense in $B$, and one can find a hyperelliptic curve with a rational Weierstrass point
	\[C\colon y^2=x^5+a_1x^4+a_2x^3+a_3x^2+a_4x+a_5\]
	such that $\Jac(C)[3]\simeq \Sigma$.
	\item\label{ob:part5} If $\Ob(B)$ has index $4$ then $B(k)$ consists of singular points and the desingularization of $B$ has no $k$-rational points at all.
\end{enumerate}
\end{restatable}

We also show that $\Ob(B)$ can indeed be of index $1$, $2$, or $4$. To this purpose we consider another classical model for the Burkhardt quartic threefold: the threefold in $\PP^5$ defined by the elementary symmetric functions in six variables $x_1,\ldots,x_6$ of degree $1$ and $4$:
\[B'\colon \sigma_1=\sigma_4=0.\]
We furthermore consider a form over $k=\RR(s,t)$ that is isomorphic to $B^{(1)}$ over $k(\sqrt{s},\sqrt{t})$, defined by
\[\begin{split}
	B''\colon&
	z_0^4 + 4z_0z_1^3 
	+ 3z_1^4 
	+ 3s^2z_2^4 
	+ 3t^2z_3^4
	+ 3s^2t^2z_4^4\\
	&+ 12sz_0z_1z_2^2 
	+ 12tz_0z_1z_3^2
	+ 12stz_0z_1z_4^2
	+ 24stz_0z_2z_3z_4
	+ 24stz_1z_2z_3z_4\\
	&- 6sz_1^2z_2^2 
	- 6tz_1^2z_3^2
	- 6stz_2^2z_3^2
	- 6stz_1^2z_4^2
	- 6s^2tz_2^2z_4^2
	- 6st^2z_3^2z_4^2=0.
\end{split}
\]
We use the notation $(a,b)$ for the class in $\Br(k)$ of a quaternion algebra over a field $k$ (see \eqref{E:quaternion_def}). We prove the following proposition in Section~\ref{S:proof_example}.
\begin{restatable}{prop}{propindex}\label{P:index}\;
\begin{enumerate}[label=(\alph*)]
\item\label{ex1} The standard model $B^{(1)}$ over $\QQ$ has $\Ob(B^{(1)})=1$, which is of index $1$.
\item\label{ex2} The symmetric model $B'\subset \PP^5$ has $\Ob(B')=(-3,-1)$, which over $\QQ$ is of index $2$.
\item\label{ex3}  The model $B''$ defined above has $\Ob(B'')=(-1,s)\otimes (-s,t)$, which over $\RR(s,t)$ has index $4$. We have $B''(k)=\{(1:-1:0:0:0)\}$, which is a singular point on $B''$. The blow-up of $B''$ at that point has no $k$-rational points at all.
\end{enumerate} 
\end{restatable}

We see that having $\Ob(B)$ of index $4$ puts severe restrictions on the rational points on $B$. As Proposition~\ref{P:index}\ref{ex3} shows, there are Kummer $3$-level structures that \emph{a priori} are admissible in the sense that they correspond to an element of $\HH^1(k,\PSp_4(\FF_3))$, but do not occur for a Kummer surface over $k$. This is in stark contrast to what happens with $3$-level structures for abelian surfaces, where Theorem~\ref{T:burkhardt_ob}\ref{ob:part4} guarantees that the corresponding moduli space is in fact unirational.

Over a number field, however, and in particular over $\QQ$, index and period of Brauer group elements agree, so $\Ob(B)$ is of index at most $2$ and we don't get a particular obstruction to $\QQ$-rational points on $B$.
 
For instance, it is not hard to find many rational points on $B'$, including ones that do not lie in $\He(B')$. We establish in Section~\ref{S:proof_density} the following.

\begin{restatable}{prop}{propdensity}\label{P:density}
	The Burkhardt quartic $B'$ over $\QQ$ is birational to the elliptic threefold in $\PP^2\times\AA^2$, defined by
	\[\begin{split}
C_{(u,v)}\colon& (u + v - 1)XY(X+Y) + (-uv + u + v)(X^2+Y^2)Z\\
&+ (-u^2 - 3uv + 3u - v^2 +
	3v - 1)XYZ  + (u^2v + uv^2 - uv)Z^3\\
&+ (u^2v - u^2 + uv^2 - 3uv + u - v^2 + v)(X+Y)Z^2=0,
\end{split}\]
where $C_{(u,v)}$ has rational flex points $(1:0:0),(0:1:0),(1:-1:0)$. The map to $B'$ is given by
\[(x_1:x_2:x_3:x_4:x_5:x_6)=(X:Y:-uZ:-vZ:Z:-X-Y+(u+v-1)Z).\]
Furthermore, $B'(\QQ)$ is Zariski-dense in $B'$.
\end{restatable}

With some modest experimentation we have not been able to find a twist of $B$ over $\QQ$ that did not have any rational points. This gives some mild circumstantial evidence for a possibly negative answer to the following question.

\begin{question}
Does there exist a Burkhardt quartic $B$ for which $\Ob(B)$ has index $2$ and for which the rational points are \emph{not} Zariski-dense?
\end{question}

This article is partially based on the masters thesis \cite{Filatov2020} of the second author written under supervision by the first.

\subsection*{Acknowledgements} We thank Frank Calegari for useful discussions, John Voight for references to Albert's work on biquaternion algebras and helpful comments and discussions, and Asher Auel for interesting discussion concerning the proof of Lemma~\ref{L:brRst}. We also thank multiple anonymous referees for various comments that helped improve the exposition in the paper.

\section{Background}

\subsection{Brauer groups and Brauer-Severi varieties}

In what follows, we frequently refer to the Brauer group $\Br(k)$ of a field $k$. There are many descriptions. It can be described as the Galois cohomology group $\Br(k)=\HH^2(k,{\ksep}^\times)$. Elements of $\Br(k)$ also correspond to $k$-isomorphism classes of \emph{Brauer-Severi} varieties: varieties that, over $\ksep$, are isomorphic to $\PP^n$ for some $n\geq 0$. We refer to \cite{GilleSzamuely2006} for details; here we just review some standard terminology and results that we need in the rest of the text.

The \emph{period} of an element in $\Br(k)$ is its order under the group structure of $\Br(k)$. We will only be dealing with elements of order dividing $2$, i.e., elements that lie in $\Br(k)[2]=\HH^2(k,\mu_2)$, where $\mu_p$ stands for the $p$-th roots of unity.

The \emph{index} of an element $\xi\in \Br(k)$ is the smallest degree of an extension $L$ such that $\xi$ lies in the kernel of the restriction map $\Br(k)\to\Br(L)$. Since a Brauer-Severi variety $V$ is isomorphic to $\PP^n$ if and only if it has a rational point, the index is also the smallest degree of an extension $L$ for which $V$ has an $L$-rational point.

The period always divides the index. For local fields and global fields, the period and index are equal. However, for fields of higher cohomological dimension, such as $\RR(s,t)$, the period can be strictly smaller than the index.

Elements of $\Br(k)$ also correspond to Brauer-equivalence classes of central simple algebras. The group law on $\Br(k)$ is induced by the tensor product on algebras. A famous theorem by Merkurjev-Suslin states that $\Br(k)[2]$ is generated by quaternion algebras. For $a,b\in k^\times$ we write $(a,b)$ for the quaternion algebra
\begin{equation}\label{E:quaternion_def}
(a,b)=k\oplus ik\oplus jk\oplus ijk, \text{ with } i^2=a,\;j^2=b,\;ij=-ji.
\end{equation}
We also use $(a,b)$ to denote its Brauer class in $\Br(k)$. The Brauer-Severi variety belonging to $(a,b)$ is the conic $Q\colon z^2-ax^2-by^2=0$.
Elements of $\Br(k)[2]$ of index at most $2$ are exactly the ones that can be represented by a single quaternion algebra, or equivalently an isomorphism class of plane conics.

\subsection{Obstructions for genus $2$ curves}

As is well-known, genus $2$ curves and, equivalently, abelian surfaces, can have different fields of moduli and fields of definition: for a non-algebraically closed field $k$, one may have an isomorphism class of genus $2$ curves over $\ksep$ that is stable under $\Gal(\ksep/k)$, but does not contain any curves defined over $k$.

This phenomenon can be made very explicit, see \cite{Mestre91}. A genus $2$ curve is geometrically determined by a degree-$6$ separated locus on a genus $0$ curve. This data can be specified over $k$ by a plane conic $Q$ and cubic curve $C$ (in fact, a $3$-dimensional linear system of cubics) over $k$. 
The data only correspond to a genus $2$ curve defined over $k$ if $Q$ is isomorphic to $\PP^1$ over $k$, in which case an appropriate genus $2$ curve is obtained as a double cover of $\PP^1$, ramified over the degree-$6$ locus. 

The isomorphism class of a conic $Q$ over $k$, a Brauer-Severi variety of dimension $1$, is an element of the Brauer group of $k$ of index and period dividing $2$.
Writing $\sM_2$ for the (coarse) moduli space of curves of genus $2$, the construction above gives rise to a map
\[\Ob\colon \sM_2(k)\to \Br(k)[2],\]
where a point $\alpha\in \sM_2(k)$ can be represented by a genus $2$ curve defined over $k$ if and only if $\Ob(\alpha)$ vanishes.

\subsection{Principally polarized abelian surfaces and their Kummer surfaces}

Let $C$ be a curve of genus $2$ over a field $k$. Then $\Jac(C)$ is a principally polarized abelian surface over $k$, and by the Torelli theorem, $C$ can be recovered from $\Jac(C)$:  there is an injective morphism between moduli spaces $\sM_2\to\sA_2$.

Associated to a principally polarized abelian surface $A$ is its \emph{Kummer surface} $\Kum(A)=A/\langle -1\rangle$, obtained by identifying points with their inverses. The fixed locus of the inversion map, the $2$-torsion $A[2]$, is $0$-dimensional of degree $16$ and maps onto the singular locus of $\Kum(A)$. In addition, one of the singular points on $\Kum(A)$ is distinguished: it is the image of the identity element of $A$.

If $A=\Jac(C)$, then $\Kum(A)$ admits a quartic model in $\PP^3$ (a Kummer quartic), and conversely, any quartic surface $K$ in $\PP^3$ with 16 nodal singularities and a distinguished node can be recognized as $\Kum(A)$ for some $A$ over $\ksep$. An explicit, classic construction to do so goes as follows.

The projective dual $K^*$ of $K$ is again a Kummer quartic surface. The 16 singularities of $K^*$ correspond to 16 \emph{tropes} on $K$: planes that intersect $K$ in a double-counting conic. Each trope passes through 6 nodes and each node lies on $6$ tropes, forming the classical $(16_6)$ \emph{Kummer configuration}.

Note that $K$ has a distinguished node and, by duality, $K^*$ has a distinguished trope. It follows that $K$ has a node defined over the base field $k$ and $K^*$ has a trope defined over $k$. It is a classical result that if $k$ is algebraically closed then $K$ is projectively isomorphic to $K^*$. For $k$ not algebraically closed, $K$ and $K^*$ may be non-isomorphic, witnessed by the fact that $K$ having a rational node does not imply that it also has a rational trope.

The distinguished trope on $K^*$ cuts out a plane conic with six marked points: the nodes the trope passes through. Equivalently, we consider the tangent cone to $K$ at the distinguished node. The six tropes passing through it intersect the tangent cone in lines through the node. Projection from the node yields a plane conic $Q_K$ with six marked points. 

If $Q_K\simeq \PP^1$ then this data determines a genus $2$ curve $C$, up to quadratic twist, such that $K=\Kum(\Jac(C))$. For a Kummer surface $K$ over $k$, we write $\Ob(K)$ for the isomorphism class of $Q_K$ in $\Br(k)$. Indeed, every point $\alpha\in \sM_2(k)$ corresponds to a Kummer quartic surface $K_\alpha$ such that $\Ob(\alpha)=\Ob(K_\alpha)$ and such that if $\Ob(\alpha)=0$, then $\Kum(\Jac(C_\alpha))=K_\alpha$. This can be checked by considering a quadratic extension $L$ of $k$ such that $\Res_L(\Ob(\alpha))=0$, construct $C_\alpha$ over $L$, and confirming that $K_\alpha$ can be descended to $k$.

In the following proposition, \emph{sufficiently general} means representing a point on an open part of the relevant moduli space. In fact the proof given works for Kummer surfaces of Picard number $17$.

\begin{prop}\label{P:kummer_generic_aut} Let $K$ be a sufficiently general quartic Kummer surface with distinguished node. Then $K$ has trivial automorphism group preserving the node.
\end{prop}
\begin{proof}
An automorphism on $K$ that fixes the distinguished node corresponds to a birational automorphism on $K^*$ that preserves a trope. One can see this by considering the desingularized Kummer surface $\tilde{K}$ that covers both $K$ and $K^*$.

Since twice a trope is a hyperplane section of $K^*$, we see that the automorphism fixes a hyperplane divisor. Since the quartic model of $K^*$ in $\PP^3$ is given by a complete linear system, we see that such an automorphism acts linearly on $K^*$, and hence also on $K$. It acts on the $16$ nodes. However, for a sufficiently general Kummer surface, the only non-identity linear transformations on the nodes are the translations (corresponding to translation by $2$-torsion), which have no fixed points.
\end{proof}

\begin{remark}\label{R:elliptic_kummer}
As is shown in \cite{CardonaQuer2005}, for a point $\alpha\in \sA_2(k)$ representing an abelian variety $A_\alpha$ over $\ksep$ with an automorphism group larger than just $\mu_2$, the variety $A_\alpha$ can be descended to a model over $k$. In that sense, $\Ob(\alpha)=1$. However, in the case where $A_\alpha$ is a product of elliptic curves, there are extra automorphisms from negation on one of the factors. As a result, $K_\alpha$ inherits a non-trivial automorphism and the isomorphism class of $K_\alpha$ over $k$ is not uniquely determined by $\alpha$. Indeed, there can be an obstruction $\Ob(K)$ in those cases that does not factor through $\Ob(\alpha)$.

We can see this in the following way. Suppose that $A$ is the Weil restriction of an elliptic curve $E$ over a quadratic extension $L=k[\sqrt{r}]$, possibly split. The quotient $E\to \PP^1$ by $-1$ induces a degree-$4$ map from $A$ to the Weil restriction $V$ of $\PP^1$. If we write
\[E^{(\delta)}\colon \delta y^2 = f(x) = x^3+a_2x^2+a_4x+a_6,\]
where $\delta\in L^\times$ prescribes a quadratic twist of $E$ over $L$, then we can model $V$ as a quadric in $\PP^3$ with affine model $x_1^2-rx_2^2-x_3=0$, related to $E^{(\delta)}$ by $x=x_1+x_2\sqrt{r}$.
Let $d=N_{L/k}(\delta)$. We can get a degree-$16$ model $K^{(d)}$ for $\Kum(A^{(\delta)})$ in weighted $9$-dimensional projective space with coordinates
\[(1:x_1:x_2:x_1^2:x_1x_2:x_2^2:x_1x_3:x_2x_3:x_3^2:w),\]
with weights $1,\ldots,1,2$. The first nine coordinates give part of the degree-$2$ Segre embedding of $V$ and the final coordinate expresses $K^{(d)}$ as a double cover of $V$ via the relation
\[dw^2=N_{L/k}(f(x_1+x_2\sqrt{r})).\]
The right hand side is indeed quartic in $x_1,x_2,x_3$ thanks to the defining relation for $V$. 
We see that the isomorphism class of $K^{(d)}$ only depends on the class of $d$ in $k^\times/k^{\times 2}$. On the other hand, we see that $K^{(d)}$ only admits a cover by $A^{(\delta)}$ if there is a non-zero solution to the norm equation $u_0^2-r u_1^2=d w^2$, i.e., if a conic is isomorphic to $\PP^1$.
\end{remark}

\subsection{Three-Level structure}\label{S:3levelstructure}

The $3$-torsion on a principally polarized abelian surface $A$ is a $0$-dimensional group scheme $A[3]$ of degree $3^4$, together with a perfect alternating pairing $A[3]\times A[3]\to \mu_3$.

Let $\Sigma$ be such a group scheme, equipped with pairing. A $3$-level structure on a principally polarized abelian surface $A$ is an isomorphism $\Sigma \to A[3]$ compatible with the pairings on either side. One such group scheme is $\Sigma^{(1)}=(\ZZ/3\ZZ)^2\times (\mu_3)^2$, with the pairing induced by the fact that $(\mu_3)^2$ is the Cartier dual to $(\ZZ/3\ZZ)^2$.

The automorphism group of $\Sigma^{(1)}$ is isomorphic to $\Sp_4(\FF_3)$. The twisting principle (see \cite{milne:etale_cohomology}*{III.4}) implies that the isomorphism class of $\Sigma$ over $k$, being a twist of $\Sigma^{(1)}$, is classified by the Galois cohomology set $\HH^1(k,\Sp_4(\FF_3))$.

We write $\overline{\Sigma}$ for $(\Sigma-\{0\})/\langle \pm 1\rangle$. It is a degree-$40$ scheme, together with a pairing $\overline{\Sigma}\times \overline{\Sigma}\to \{0,1\}$, determined by whether the corresponding representatives in $\Sigma$ pair trivially. The automorphism group $\Sp_4(\FF_3)$ acts on $\Sigmabar$, with the center (generated by $-1$) acting trivially. Hence the action factors through the simple group $\PSp_4(\FF_3)$.

\begin{remark}
The pairing information on $\Sigmabar$ can be almost captured by an associated incidence structure. We can view $\Sigmabar$ as the set of $1$-dimensional subspaces of $V=(\FF_3)^4$, where $V$ is equipped with a perfect alternating pairing. There are exactly $40$ maximal isotropic subspaces of $V$, and each of those contains exactly four $1$-dimensional subspaces. Hence $\Sigmabar$ comes with $40$ subsets of cardinality $4$. The group of permutations preserving this incidence structure is $\PGSp_4(\FF_3)$, which contains $\PSp_4(\FF_3)$ as an index-$2$ subgroup.
\end{remark}

Let $A$ be an abelian surface with a $3$-level structure $\Sigma\to A$. Since multiplication-by-3 commutes with negation, it induces a well-defined map $K\to K$, which we call \emph{pseudo-multiplication by $3$}. The $3$-level structure on $A$ induces an isomorphism between $\Sigmabar\cup{0}$ and the fibre of the distinguished point. We call this a Kummer $3$-level structure on $K$.
A Kummer surface has pseudo-multiplication maps, regardless of whether $\Ob(K)$ is trivial. Hence we have a Kummer $3$-level structure on it as well.

The possible Kummer $3$-level structures are classified by $\HH^1(k,\PSp_4(\FF_3))$. Taking cohomology of the short exact sequence
\[1\to\mu_2\to\Sp_4(\FF_3)\to\PSp(\FF_3)\to 1\]
gives a map
\[\HH^1(k,\Sp_4(\FF_3))\to \HH^1(k,\PSp_4(\FF_3))\stackrel{\Ob}{\longrightarrow} \HH^2(k,\mu_2)\]
For a Kummer $3$-level structure $\Sigmabar$ we write $\Ob(\Sigmabar)$ for its class in $\HH^2(k,\mu_2)$.

By Proposition~\ref{P:kummer_generic_aut}, we see that for a point $\alpha\in\sA_2(k)$ for which $K_\alpha$ is sufficiently general, we have that $K_\alpha$ and $\Sigmabar$ are determined by $\alpha$ and hence that $\Ob(K_\alpha)=\Ob(\Sigmabar)$. 

As explained in Remark~\ref{R:elliptic_kummer}, for $\alpha$ that correspond to products of elliptic curves, the isomorphism class of the Kummer surface is only determined up to twist, but the Kummer $3$-level structure determines the twist $K_{\alpha,\Sigmabar}$. In this case, it can be checked that once again, $\Ob(K_{\alpha,\Sigmabar})=\Ob(\Sigmabar)$.

\subsection{Moduli spaces with $3$-level structure}

We write $\sA_2(\Sigma)$ for the moduli space of principally polarized abelian surfaces $A$ with $3$-level structure $\Sigma\to A$. 
Since the $(-1)$-automorphism on $\Sigma$ is the restriction of $-1$ on $A$, the isomorphism class of $\sA_2(\Sigma)$ only depends on $\Sigmabar$, and therefore we write $\sA_2(\Sigmabar)$. 

Indeed, from
\[\HH^1(k,\mu_2)\to \HH^1(k,\Sp_4(\FF_3))\to \HH^1(k,\PSp_4(\FF_3))\]
we see that different $\Sigma,\Sigma'$ map to isomorphic $\Sigmabar$ exactly when they are quadratic twists, and taking quadratic twists of an abelian variety $A$ will correspondingly twist its $3$-torsion.

In particular, we see that the level-structure-forgetting morphism $\sA_2(\Sigmabar)\to \sA_2$ is Galois with automorphism group $\PSp_4(\FF_3)$. Hence, for \emph{any} class $\Sigmabar$ in $\HH^1(k,\PSp_4(\FF_3))$ we have a corresponding twist $\sA_2(\Sigmabar)$, which we can then consider as a moduli space of Kummer surfaces with $3$-level structure.
The space comes with an obstruction map
\[\Ob\colon B(k)\to \Br(k);\; \alpha\mapsto\Ob(K_\alpha),\]
which is constant $\Ob(\Sigmabar)$.

\section{Representation-theoretic description of the Burkhardt quartic threefold}

The Burkhardt quartic form can be characterized, and in fact was discovered as, the unique quartic invariant of a certain five-dimensional representation of $\Sp_4(\FF_3)$. The degree-$6$ unirational parametrization $\PP^3\dashrightarrow B^{(1)}$ was similarly studied in classical representation-theoretic terms, see \cites{Klein1888, Witting1887, Maschke1889, Burkhardt1891, Coble1917}. 

We review some of these results below, emphasizing the fact that these can all be characterized in terms of low-dimensional irreducible representations of $\PSp_4(\FF_3)$. This means they can all be recovered from different forms over non-algebraically closed base fields as well.

\subsection{The Burkhardt quartic through representation theory}\label{S:representation_theory}

Let $k$ be a field of characteristic different from $2,3,5$. For now, we fix the $3$-level structure $\Sigma=\Sigma^{(1)}=(\ZZ/3\ZZ)^2\times(\mu_3)^2$ and the finite algebraic group $\Gamma$, defined over $k$, that describes its automorphism group. Then $\Gamma(\ksep)$ is isomorphic to $\Sp_4(\FF_3)$, but it comes with a specific Galois action. We refer to $\Gamma$ as a \emph{form} of $\Sp_4(\FF_3)$ over $k$.

The group affords two faithful irreducible $4$-dimensional representations: $\rho_4$ and its dual $\rho_4^\vee$ the complex conjugate. Explicit generators for the image of $\rho_4$ are
\[\begin{aligned}
A_1&=
\begin{pmatrix}
	-1 & 0 & 0 & 0 \\
	0 & -1 & 0 & 0 \\
	0 & 0 & -1 & 0 \\
	0 & 0 & 0 & -1
\end{pmatrix},\\
A_2&=\frac{1}{3}\begin{pmatrix}
	3\zeta & 0 & 0 & 0 \\
	0 & \zeta + 2 & \zeta + 2 & -\zeta + 1 \\
	0 & \zeta + 2 & \zeta - 1 & -\zeta - 2 \\
	0 & -\zeta + 1 & -\zeta - 2 & \zeta + 2
\end{pmatrix},\\
A_3&=\frac{1}{3}\begin{pmatrix}
	2\zeta + 1 & -2\zeta - 1 & 0 & \zeta + 2 \\
	\zeta - 1 & 2\zeta + 1 & 0 & -\zeta + 1 \\
	0 & 0 & 3\zeta + 3 & 0 \\
	\zeta + 2 & 2\zeta + 1 & 0 & 2\zeta + 1
\end{pmatrix}.
\end{aligned}\]
Note that the image of $\rho_4$ is stable under complex conjugation, so the image is a form of $\Sp_4(\FF_3)$ over $k$, even though the elements are not all individually defined over $k$.

\begin{table}
\[\begin{array}{c|rrr}
	\chi_i&\chi_i(A_1)&\chi_i(A_2)&\chi_i(A_3)\\
	\hline
	\rho_4&	-4 & 2\zeta + 1 & 3\zeta + 2 \\
	\rho_4^\vee&	-4 & -2\zeta - 1 & -3\zeta - 1 \\
	\rho_5&	5 & 0 & -3\zeta - 1 \\
	\rho_5^\vee&	5 & 0 & 3\zeta + 2 \\
	\rho_{10}&	10 & -1 & -3\zeta - 5 \\
	\rho_{10}^\vee&	10 & -1 & 3\zeta - 2 \\
	\chi_7&	-20 & 0 & 7 \\
	\rho_{20}&	20 & 1 & 2 \\
	\chi_9&	-20 & 0 & 3\zeta - 5 \\
	\chi_{10}&	-20 & 0 & -3\zeta - 8 \\
	\chi_{11}&	-20 & 2\zeta + 1 & 6\zeta + 1 \\
	\chi_{12}&	-20 & -2\zeta - 1 & -6\zeta - 5 \\
	\chi_{13}&	30 & -1 & 3 \\
	\rho_{30}&	30 & 0 & -9\zeta - 6 \\
	\rho_{30}^\vee&	30 & 0 & 9\zeta + 3
\end{array}\]
\caption{Irreducible characters of $\Sp_4(\FF_3)$ of degrees $4,5,10,20,30$}
\label{T:char}
\end{table}

The representation $\rho_4$ induces a projective representation of $\Gammabar=\Gamma/\langle -1\rangle$, whose image is a form of $\PSp_4(\FF_3)$ over $k$. The image of the projective representation can be described as the automorphism group of a \emph{Witting configuration} in $M=\PP^3$, consisting of $40$ points and $40$ planes, such that each point lies in $12$ planes and each plane passes through $12$ points.

Table~\ref{T:char} lists some irreducible characters of $\Sp_4(\FF_3)$, together with the character values at the generators. The values at these generators happen to separate all characters of $\Sp_4(\FF_3)$. Of particular note is that the characters of degree up to $10$ are almost completely determined by their degree: there are at most two of each, in which case one is the complex conjugate of the other.

\setcounter{footnote}{1}%start footnote counter at "dagger" since that's a less used math symbol than "*" and therefore is less likely to be confused with math notation
A simple character computation, or an explicit computation with the given generators $A_1,A_2,A_3$, shows that $\Sym^2\rho_4=\rho_{10}$ 
and that $\Sym^4\rho_4=\rho_5\oplus\rho_{30}$.
The representation $\rho_5$ gives rise to a rational map $\pi\colon M\dashrightarrow \PP^4$, defined by a linear system\footnote{Recall that the dimension of a linear system is a \emph{projective} dimension and hence one less than the dimension of the vector space corresponding to it.}
of degree $4$ and dimension $4$ that vanishes on the Witting configuration in $M$. For $\Gamma$ as defined above, we can write $(t_1:t_2:t_3:t_4)$ for the coordinates on $M$. The linear system corresponding to $\rho_5$ as a component of $\Sym^4\rho_4$ is generated by
\[\begin{aligned}
	Y_0&=	3t_1t_2t_3t_4,\\
	Y_1&=	t_1(t_2^3 + t_3^3 - t_4^3),\\
	Y_2&=	-t_2(t_1^3 + t_3^3 + t_4^3),\\
	Y_3&=	t_3(-t_1^3 + t_2^3 + t_4^3),\\
	Y_4&=	t_4(t_1^3 + t_2^3 - t_3^3)
\end{aligned}
\]
With this description of $\pi$, the image in $\PP^4$ is dense in
\[B\colon y_0 (y_0^3 + y_1^3 + y_2^3 + y_3^3 + y_4^3 ) + 3y_1 y_2 y_3 y_4=0.\]
Hence, we recover the Burkhardt quartic threefold.

Of special note is that the centre of $\Gamma$ lies in the kernel of $\rho_{5}$ and $\rho_{10}$, so these representations are inflations of representations of $\Gammabar$. This allows us to obtain a model of $\pi\colon M\dashrightarrow B$ entirely in terms of the representation theory of $\Gammabar$ in the following way.

If we consider the degree-$2$ Veronese embedding $v_2\colon \PP^3\to\PP^9$ then the action of $\Gamma$ on $\PP^9$ is through $\Sym^2\rho_4=\rho_{10}$. We have $\Sym^2\rho_{10}=\rho_5\oplus \rho_{30} \oplus \rho_{20}$. Indeed, the image of $v_2$ is defined by a $20$-dimensional space of quadrics on $\PP^9$, corresponding to $\rho_{20}$ and the rational map to $B$ can now be obtained as a linear system of degree $2$ and dimension $4$ on $\PP^9$, from $\rho_5$ as a constituent of $\Sym^2\rho_{10}$.

Conversely, we can recover $\rho_{10}$ from $\rho_5$ via $\bigwedge^{\!2} \rho_5 =\rho_{10}^\vee$. That means that for any twist of $B$ in $\PP^4$, together with its automorphism group $\Gammabar$, given by its $5$-dimensional representation $\rho_5$ on the coordinates of $B$, we can recover $\rho_{10}$, a model of $M\subset \PP^{9}$ through $\rho_{20}$, and the map $M\dashrightarrow B$.

\subsection{The moduli interpretation of $B$ and $M$}\label{S:moduli_construction}

We recall that the \emph{polars} of a form $F\in k[x_1,\ldots,x_n]$ at a point $(\alpha_1,\ldots,\alpha_n)$ are defined by
\[P^{(1)}_{\alpha}(F)=\sum_{i=1}^n \alpha_i \frac{\partial F}{\partial x_i}\text{\; and\; }P^{(r+1)}_{\alpha}(F)=P^{(r)}_{\alpha}(P^{(1)}_{\alpha}(F))\text{ for } r=1,2,\ldots.\]
We write $P^{(i)}_\alpha$ for the polars of the form defining $B$. These are forms of degrees $3,2,1$ for $i=1,2,3$ respectively.

Intersection of $B$ with its hessian $\He(B)$ yields a locus that over $k^\sep$ consists of $40$ planes, called $j$-planes. The action of $\Gammabar$ on these is conjugate to the action of $\PSp_4(\FF_3)$ on the cyclic subgroups of $\Sigma$ of order $3$. These subgroups are in bijection with $\Sigmabar$. The pairing information is also reflected in the $j$-plane configuration: planes that pair trivially meet in a line and others meet in a point.

There is a synthetic description of the modular interpretation of $B$, see \cites{Coble1917,Hunt96}. Let $\alpha$ be a point in $B\setminus \He(B)$.  Then $P^{(3)}_\alpha\cap P^{(2)}_\alpha$ is a cone over a plane conic $Q_\alpha$, and the cubic $P^{(1)}_\alpha$ cuts out a degree six locus on $Q_\alpha$. In fact, the enveloping cone at $\alpha$ of $P^{(1)}_\alpha$ yields a cone over a dual Kummer surface $K_\alpha^*$, with $P^{(3)}_\alpha$ projecting to the distinguished trope. The $j$-planes project to tangent planes of $K_\alpha^*$ and hence yield points on its dual $K_\alpha$, marking a Kummer 3-level structure $\Sigmabar$ on $K_\alpha$. It follows that $\Ob(\alpha)=\Ob(\Sigmabar)$.

The rational map $\pi\colon M\dashrightarrow B$ has generic degree $6$ and also has a modular interpretation: outside $\He(B)$ it corresponds to the choice of an odd theta-characteristic, or, equivalently, a Weierstrass point on the genus $2$ curve of which $A_\alpha$ is the Jacobian. In recognition of the work Maschke did on these spaces \cite{Maschke1889}, the space $M$ is sometimes referred to as the \emph{Maschke $\PP^3$}.

As discussed before, for each $\xi \in \HH^1(k,\Gammabar)$ we get a different Kummer $3$-level structure $\Sigmabar^{(\xi)}$, and a corresponding form $\Gammabar^{(\xi)}$. By Hilbert 90, any representation $\rho\colon \Gammabar \to \GL_n$ gives rise to a corresponding representation $\rho^{(\xi)}\colon \Gammabar^{(\xi)} \to \GL_n$.

That means that $\Gammabar^{(\xi)}$ affords representations corresponding to $\rho_5,\rho_{10},\rho_{20}$. In particular, we get a twist $B^{(\xi)}\subset\PP^4$, together with a $3$-dimensional Brauer-Severi variety $M^{(\xi)}\subset\PP^9$ and a 
degree-$6$ rational map $\pi^{(\xi)}\colon M^{(\xi)}\to B^{(\xi)}$.
Note that $M^{(\xi)}$ is isomorphic to $\PP^3$ precisely when the action of $\Gammabar^{(\xi)}$ can be lifted to a $4$-dimensional linear representation, i.e., when $\xi$ can be lifted to $\HH^1(k,\Gamma)$. It follows that the isomorphism class of $M^{(\xi)}\subset\PP^9$ as a Brauer-Severi variety is the image $\Ob(\Sigmabar^{(\xi)})=\Ob(\xi)\in \HH^2(k,\mu_2)$.

\subsection{Proof of Theorem~\ref{T:burkhardt_rep}}
\label{S:proof_burkhardt_rep}
We restate Theorem~\ref{T:burkhardt_rep} and give its proof.

\thmburkhardtrep*

\begin{proof}
\begin{description}
	\item[Part \ref{rep:part1}] \footnote{Alternatively, as Anastasia Vikulova points out, one can also establish that $B^{(1)}$ is the anticanonical model of its desingularization, implying that $B$ admits a quartic model in the same way.} As described above, the automorphism group $\Aut(B)$ of $B$ is some form of $\PSp_4(\FF_3)$. Thus, by Hilbert 90, there is a representation $\rho_{10}$ of $\Aut(B)$, giving a linear projective action of $\Aut(B)$ on $\PP^9$. Similarly, the decomposition $\Sym^2\rho_{10}=\rho_5\oplus\rho_{30}\oplus\rho_{20}$ yields an $\Aut(B)$-stable $19$-dimensional linear system of quadrics on $\PP^9$ defining a $3$-dimensional Brauer-Severi variety $M\subset\PP^9$, together with a covariant map $\pi\colon M\dashrightarrow\PP^4$ from $\rho_5$. Its image then yields a quartic model for $B$. Over $\ksep$, this model differs from $B^{(1)}$ by a linear transformation, so $B$ also has a singular locus of dimension $0$ and degree $45$.
	\item[Part \ref{rep:part2}] We have already constructed the map $\pi$ above. Base changing to $\ksep$ does not change its generic degree, and there it agrees with the standard expression of the Maschke $\PP^3$ over $B^{(1)}$.
	\item[Part \ref{rep:part3}] By definition, $\Ob(B)$ is the class of the Brauer-Severi variety $M$. From the cohomological description, it is clear that $\Ob(B)\in \HH^2(k,\mu_2)=\Br(k)[2]$, so its period divides $2$. Since $M$ is a Brauer-Severi variety of dimension $3$, its isomorphism class is represented by an element of $\HH^1(k,\PGL_4(\ksep))$, which also classifies $16$-dimensional central simple algebras. Such an algebra is split by an extension of degree dividing $4$. \qedhere
\end{description}
\end{proof}

\begin{remark}\label{R:singmod}
The moduli interpretation of $B$ extends to products of elliptic curves as well: the blow-up of each of the $45$ nodal singularities of $B$ yields a component of the locus of $A_2(3)$ corresponding to products of elliptic curves. Indeed, $\PSp_4(\FF_3)$ has a unique conjugacy class of index $45$ subgroups, which are the stabilizers of decompositions of its standard representation into non-isotropic $2$-dimensional subspaces.

The tangent cone of a node $s$ on $B$ is a cone over a non-singular quadric $V\subset \PP^3$. Each node lies on eight $j$-planes, which map to four lines of each ruling on $V$. A choice of point $\alpha$ on $V$ marks a distinguished point and one line from each ruling by intersection with the tangent plane. Each $4$-tuple of lines cut out by $j$-planes cuts out the locus of a $3$-division polynomial on one of the lines which, together with the marked intersection point, determines a $\ksep$-isomorphism class of an elliptic curve. As explained in Remark~\ref{R:elliptic_kummer}, this determines an elliptic Kummer surface up to twist, and $\Sigmabar$ somehow encodes which twist. We have not found a direct way of reading off the full information of the Kummer surface in this situation, but on general principles we know its obstruction will be $\Ob(\Sigmabar)$.
\end{remark}

\section{Period-index questions about obstructions}

\subsection{Proof of Theorem~\ref{T:burkhardt_ob}}
\label{S:proof_burkhardt_ob}

We restate Theorem~\ref{T:burkhardt_ob} and give its proof.

\thmburkhardtob*

\begin{proof}
\begin{description}
	\item[Part \ref{ob:part1}] Note that the isomorphism classes of Burkhardt quartics as well as of Kummer $3$-level structures are classified by $\HH^1(k,\PSp_4(\FF_3))$. The synthetic description described in Section~\ref{S:moduli_construction} gives a way, given a point $\alpha$ on $B\setminus \He(B)$, to construct a Kummer surface $K_\alpha$ with the requisite level structure. In particular, we can do so at the generic point, to get a universal family over an open part of $B$.
	\item[Part \ref{ob:part2}] As noted in Section~\ref{S:3levelstructure}, for a Kummer $3$-level structure $\Sigmabar$ on a quartic Kummer surface, we have $\Ob(\Sigmabar)=\Ob(K)$. Furthermore, we have $\Ob(B)=\Ob(\Sigmabar)$.
	\item[Part \ref{ob:part3}] By \cite{BruinNasserden2018}*{Proposition~2.8}, the choice of $j$-plane allows the construction of a cubic genus $1$ curve together with a cubic map to $\PP^1$, such that $C_\alpha$ is the discriminant curve of the cubic extension. This directly determines a model of $C_\alpha$ over $k$, so there is no obstruction for $C_\alpha$ and therefore $\Ob(B)=1$.
	\item[Part \ref{ob:part4}] Note that $\Ob(\Sigmabar)=\Ob(B)$ is the class of the Maschke $M$ associated to $B$. Hence, if it is trivial then $M\simeq \PP^3$, and $\pi\colon \PP^3\to B$ yields a unirational map. The image $\pi(\PP^3(k))$ is then Zariski-dense. Points in the image correspond to Jacobians of genus $2$ curves with a marked Weierstrass point, i.e., curves that admit a quintic affine model.
	\item[Part \ref{ob:part5}] If $\alpha \in B(k)\setminus \He(B)(k)$, then $\Ob(B)=\Ob(K_\alpha)$ is represented by a conic $Q_\alpha$, and therefore of index at most $2$. Hence, if $\Ob(B)$ is of index $4$ then any rational point on $B$ must lie in $\He(B)$. By \ref{ob:part3} we know that for the field of definition $L$ of any $j$-plane, the restriction of $\Ob(B)$ to $L$ is trivial. If the index of $\Ob(B)$ is $4$, then it follows that $L$ has degree at least $4$ and hence that any rational point on $\He(B)$ must lie on at least four $j$-planes, the conjugates. But the only points that lie on more than two $j$-planes are the singular points of $B$. Furthermore, by Remark~\ref{R:singmod} we see that the special fibre of the blow-up of $B$ at any one of these singularities has a modular interpretation as well. By Remark~\ref{R:elliptic_kummer} we see that any rational point on it would lead to a representative of $\Ob(B)$ of index at most $2$, which would contradict that its index is $4$. \qedhere
\end{description}
\end{proof}

\subsection{Proof of Proposition~\ref{P:index}}
\label{S:proof_example}

We restate Proposition~\ref{P:index} and give its proof.

\propindex*

\begin{proof}	
Part~\ref{ex1} follows because over finite fields $\Br(k)=0$ and for infinite fields there are abelian varieties with $3$-torsion structure $(\ZZ/3\ZZ)^2\times (\mu_3)^2$.

For Part~\ref{ex2} we let $\sigma_1,\sigma_4\in k[x_1,\ldots,x_6]$ be the elementary symmetric functions of degrees $1$ and $4$ respectively. Then $B'\colon \sigma_1=\sigma_4=0$ is also a Burkhardt quartic threefold, lying in the hyperplane $\sigma_1=0$ inside $\PP^5$. We have $\alpha=(40:-30:-8:-5:3:0)\in B'(\QQ)$ and $Q_\alpha$ is isomorphic to the plane conic $3x^2+y^2+z^2=0$. This conic is not isomorphic to $\PP^1$ over $\QQ$.

For Part~\ref{ex3}
we take $k=\RR(s,t)$, a bivariate function field. We take a twist of $B^{(1)}$ that is isomorphic to $B^{(1)}$ over $k(\sqrt{s},\sqrt{t})$, by setting
\[\begin{aligned}
y_0&=z_0\\
y_1&=z_1+z_2\sqrt{s}+z_3\sqrt{t}+z_4\sqrt{st}\\
y_2&=z_1-z_2\sqrt{s}+z_3\sqrt{t}-z_4\sqrt{st}\\
y_3&=z_1+z_2\sqrt{s}-z_3\sqrt{t}-z_4\sqrt{st}\\
y_4&=z_1-z_2\sqrt{s}-z_3\sqrt{t}+z_4\sqrt{st}.
\end{aligned}
\]
This yields the model $B''$ as stated.

Note that $\Ob(B'')$ is an element of period $2$ and that it trivializes upon base change to $k(\sqrt{s},\sqrt{t})$. Also note that $B''$ is actually defined over $k[s,t]$ and has good reduction outside $st=0$. We write $R=\RR[s,s^{-1},t,t^{-1}]$.  We see that $\Ob(B)\in \Br(R)$. Since $R$ is a regular domain with fraction field $k$, base extension gives a natural injection $\Br(R)\to \Br(k)$ and we identify $\Br(R)$ with its image in $\Br(k)$.

\begin{lemma}\label{L:brRst}
The group $\Br(R)[2]$ is generated by \[(-1,-1),(-1,s),(-1,t),(s,t).\]
\end{lemma}
\begin{proof}
For a field $K$ of characteristic $0$ we first establish the structure of $\Br(K[t,t^{-1}])$. We first note that $\Br(K[t])=\Br(K)$ by \cite{AuslanderGoldman1960}*{Proposition~7.7}. The residue map at the prime ideal $(t)$ then gives us the exact sequence
\[0\to\Br(K[t])\to \Br(K[t,t^{-1}])\to K^\times/K^{\times 2},\]
which is split by the map $K^\times \to \Br(K[t,t^{-1}])[2]$ defined by $a\mapsto (a,s)$. It follows that a class in $\Br(K[s,s^{-1}])[2]$ can be represented by
\[A =  B \otimes (a,t) \text{ with } B\in\Br(K)[2] \text{ and }a\in K^\times\]

We apply this result for $K=\RR(s)$ to see that classes in $\Br(R)\subset\Br(\RR(s)[t,t^{-1}])$ can be represented by
\[A =  B \otimes (a,t) \text{ with } B\in\Br(\RR(s))[2] \text{ and }a\in \RR(s)^\times\]
Note that evaluation at $t=1$ yields a homomorphism $\Br(R)\to \Br(\RR[s,s^{-1}])$ which sends the class of $A$ to that of $B$, so in fact we see that $B$ must represent a class in $\Br(\RR[s,s^{-1}])$, which, using the result on Laurent polynomial rings, yields that
\[ A = C \otimes (b,s) \otimes (a,t) \text{ with } C\in\Br(\RR)[2], b\in \RR^\times, a\in \RR(s)^\times.\]
Without loss of generality we can take $a\in \RR[s]$ and square-free. Let $p$ be an irreducible factor of $a$. If $(p)\neq(s)$ then $A$ must be Azumaya at $(p)$. Let $L=\RR(s)[t]/(p)$ be the corresponding residue field. The residue map $\Br(\RR(s,t))[2]\to L^\times/L^{\times2}$ sends $A$ to $t$ or $st$, depending on whether $p$ divides $b$. Neither is a square, so it follows that $a=cs$ or $a=c$ for some $c\in\RR$. The statement now follows from the fact that $\Br(\RR)=\langle (-1,-1)\rangle$ and $\RR^\times/\RR^{\times 2}=\langle -1\rangle$.
\end{proof}

%\begin{proof}
%We consider $X=\Spec\RR[s,t]$, the divisors $D_s,D_t$ defined by the vanishing of $s,t$ respectively, and the open subscheme $U=\Spec\RR[s,s^{-1},t,t^{-1}]$ obtained by removing $D_s,D_t$. Then $\Br(R)=\Br(U)$. We link this to simpler objects by \cite{CTSkorobogatov2019}*{Theorem~3.7.2}, which when restricted to $2$-torsion, yields the exact sequence
%\[1\to\Br(X)[2]\to \Br(U)[2] \to \HH^1(D_s,\mu_2)\oplus \HH^1(D_t,\mu_2).\]
%By \cite{AuslanderGoldman1960}*{Proposition~7.7} we have
%\[\Br(X)=\Br(\RR[s,t])=\Br(\RR[s])=\Br(\RR).\]
%Since $D_s$ and $D_t$ are affine lines, we have $\HH^1(\AA^1_\RR,\mu_2)=\HH^1(\Spec(\RR),\mu_2)$ which is $\RR^\times/\RR^{\times2}$ by Kummer theory.
%The group $\Br(\RR)$ is of order two, generated by $(-1,1)$, and $\RR^\times/\RR^{\times 2}=\langle -1\rangle$ is of order two as well.
%This means that $\Br(R)[2]$ is of order $8$ at most. The quaternion algebras representing the classes stated are obviously Azumaya over $R$ and from their ramification loci over $X$, also independent. 
%\end{proof}

We note that the product of restriction maps
\[\Br(k)\to \Br(k(\sqrt{s}))\times\Br(k(\sqrt{t}))\times\Br(k(\sqrt{st}))\]
is injective on $\Br(R)$.
We compute the restriction to $\Br(k(\sqrt{s}))$ by specializing to $s=1$. The intersection with $z_3=z_4=0$ yields a genus $0$ curve on $B''$, with the point $\alpha=(16:-31:9:0:0)$ outside $\He(B'')=0$. We find that $Q_\alpha$ is equivalent to $X^2-tY^2+Z^2=0$, and therefore that the restriction of $\Ob(B'')$ to $k(\sqrt{s})$ is $(-1,t)$. Symmetry gives us that specializing $t=1$ should yield $(-1,s)$ and specializing $st=1$ should yield $(-1,s)=(-1,t)$. Computation shows
$\Ob(B'')=(-1,t)\otimes(-1,s)\otimes(s,t)=(-1,s)\otimes(-s,t)$. We describe two ways to verify that this class is of index four.

First, one can simply enumerate all the classes of index at most two, since they will be of the form $(a,b)$, where $a,b$ lie in the multiplicative group generated by $\{-1,s,t\}$. Given that $(a,a)=(-1,a)$, we see there are $\binom{7}{2}+1$ choices, but many represent equivalent classes. The classes that are not covered (and hence must be of index four) are
\[(-1,-1)\otimes (s,t),\; (-1,-1)\otimes (s,-t),\; (-1,-1)\otimes (-s,t),\; ( -1,s)\otimes(-s,t).\]
Alternatively, one can use that a biquaternion algebra is of index four if and only if its \emph{Albert form} is anisotropic, see \cite{Lam2005}*{Albert's Theorem 4.8}. In fact, Albert's original example \cite{Albert1932}*{Theorem~1} applies directly to $(s,s)\otimes(t,st)$, which is equivalent to our algebra. 

It follows from Theorem~\ref{T:burkhardt_ob}\ref{ob:part5} that all rational points of $B''$ lie in the singular locus. The only rational point there is $(1:-1:0:0:0)$.

In order to show that the desingularization of $B''$ has no rational points at all, we can also directly look at the blow-up. The tangent cone to $B''$ at $(1:-1:0:0:0)$ is the affine cone over the quadric 
\[su_0^2+tu_1^2+stu_2^2-3u_3^2=0,\]
which is indeed easily checked to have no non-zero solutions over $\CC[[s,t]]$: a non-zero solution can be reduced to a solution with coprime coordinates, but by solving for coefficients successively, one finds that all $u_i$ must be divisible by $s$ and $t$. \qedhere
\end{proof}

\subsection{Proof of Proposition~\ref{P:density}}\label{S:proof_density}
We restate Proposition~\ref{P:density} and give its proof.

\propdensity*

\begin{proof}
First note that $B'$ has $15$ rational singularities, constituting the orbit of $(1:-1:0:0:0:0)$ under the action of $S_6$ on the coordinates. These singularities form $20$ triples of collinear points. The lines lie in $\He(B')\cap B'$. For instance, the singularities $(1:-1:0:0:0:0),(1:0:0:0:0:-1),(0:1:0:0:0:-1)$ lie on the line $L_{345}=x_3=x_4=x_5=\sigma_1=0$. We consider the $2$-dimensional linear system of planes $V_{u,v}$ in $\sigma_1=0$ containing this line, defined by
\[V_{u,v}\colon x_3-ux_5=x_4-vx_5=\sigma_1=0.\]
The intersection $V_{u,v}\cap B'$ decomposes into the line $L_{345}$ and the plane cubic $C_{u,v}$ stated in the proposition. It is straightforward to check that the singularities give rise to three collinear flexes on $C_{u,v}$. By choosing one of those flexes as zero-section, we see that $C_{u,v}$ is an elliptic threefold with $3$-torsion. This yields a birational elliptic fibration on $B'$. In fact, the different choices of triples of collinear singularities on $B'$ give us $20$ such fibrations. In order to establish density of rational points on $B'$ we use a standard trick combining these multiple fibrations. We intersect $B'$ with increasing linear spaces and establish density in each iteratively.

It is straightforward to check that $V_{3/5,4}$ passes through the point $P_0=(20: 2: -9: -60: 15: 32)\in B'(\QQ)$ and that it yields a
non-torsion point on $C_{3/5,4}$. Since the rational points on the line $x_3=x_4=x_5=\sigma_1=0$  are definitely dense, this gives that rational points on $B'$ are dense in the intersection of $B'$ with the plane spanned by $P_0$ and $L_{345}$. 

Next, we pick the fibration generated by planes through $L_{245}\colon x_2=x_4=x_4=\sigma_1=0$.
The fibers that intersect the plane above form a one-dimensional family of elliptic curves.
The density established above yields infinitely many members in this family that have a rational point besides those arising from the three torsion sections. In only finitely many of those can this be torsion, so most of those fibers have infinitely many rational points themselves. This yields Zariski density of rational points in $B'$ intersected by the $3$-space spanned by $P_0$, $L_{345}$ and $L_{245}$.

We repeat this trick once more using the fibration generated by planes through $L_{145}\colon x_1=x_4=x_5=\sigma_1$. The multi-section obtained by intersecting with the $3$-space above is generically non-torsion, so there is a proper sublocus where it reduces to torsion of order, say, at most $12$ (the largest that can occur over $\QQ$). The result above shows there is a Zariski-dense set of fibers that have an extra rational point arising from this multi-section, and it follows a Zariski-dense subset has positive rank. This yields Zariski-density of rational points on $B'$.
\qedhere
\end{proof}

\section{Computational considerations}

In this section we collect some remarks concerning explicit construction of the various objects considered. First, for a Burkhardt quartic $B\subset \PP^4$ over a field $k$ of characteristic different from $2,3,5$ and a point $\alpha\in B(k)$ outside the Hessian locus, we can readily construct a conic $Q_\alpha$ representing $\Ob(B)\in \Br(k[2])$: as discussed in Section~\ref{S:moduli_construction}, the intersection of polars $P_\alpha^{(2)}\cap P_\alpha^{(3)}$ yields a cone over a plane conic $Q_\alpha$ and $P_\alpha^{(1)}$ cuts out a degree six locus on $Q_\alpha$. The conic represents $\Ob(B)$ and the degree six locus marked on it allows the recovery of $C_\alpha$, if the obstruction is trivial.

More directly, the enveloping cone at $\alpha$ of $P_\alpha^{(1)}$ is a cone over $K_\alpha^*$, with the Kummer $3$-level structure marked on it by $\He(B)$, as described in Section~\ref{S:moduli_construction} as well.

In order to explicitly recover $\pi\colon M\dashrightarrow B$ from an explicitly given Burkhardt quartic $B\subset \PP^4$ one can determine $\Gammabar\subset \SL_5(k)$ that preserves the singular locus. This will be isomorphic to $\PSp_4(\FF_3)$ as a group and we have a natural $5$-dimensional $k$-vectorspace $V_5$ with an action of $\Gammabar$. Following the remarks in Section~\ref{S:representation_theory}, we consider $V_{10}=(\bigwedge^{\!2}V_5)^\vee$. Then $\PP(V_{10})$ yields the natural projective space $\PP^9$ containing $M$. We have that $\Sym^2 V_{10}$ decomposes as a $\Gammabar$-module into a $20$-dimensional space of quadratic forms on $\PP^9$ defining $M$, a $5$-dimensional space of quadratic forms yielding the map $\pi\colon M\dashrightarrow B$, and a remaining $30$-dimensional space.

The general representation theory of $\Sp_4(\FF_3)$ and $\PSp_4(\FF_3)$ implies that $M\subset \PP^9$ is a three-dimensional Brauer-Severi variety and that its class in $\Br(k)$ is the same as $\Br(B)$, but reading off important characteristics such as period and index may not be so simple from this explicit representation of $M$ as a variety in $\PP^9$.

Finally, given some explicit representation of $\xi\in \HH^1(k,\PSp_4(\FF_3))$, Theorem~\ref{T:burkhardt_rep}\ref{rep:part1} asserts the existence of a corresponding quartic model $B^{(\xi)}\subset \PP^4$. Whether it is easy to construct $B$ depends on how $\xi$ is specified. If we can somehow get the representation $\rho_5$ from it, then the corresponding Burkhardt quartic form is just the unique quartic invariant form (up to scaling). Generally, specifying just the splitting field of $\Gammabar$, for instance by specifying a degree-$40$ algebra, does not fully capture $\xi$, because it does not specify the pairing information.

There is one case where simply specifying an algebra is sufficient: there is a unique subgroup conjugacy class in $\PSp_4(\FF_3)$ of order $720$, represented by $\Sym(6)$ acting by permutation on the coordinates of $B'$.
For a square-free degree-$6$ polynomial
\[h(T)=(T-\beta_1)\cdots(T-\beta_6)\]
we set
\[\tilde{x}_i=x_1+\beta_ix_2+\cdots+\beta_i^5x_6\text{ for }i=1,\ldots,6.\]
It follows that a permutation on the $\beta_i$ has the same permutation action on the $\tilde{x}_i$. Furthermore, we see that the elementary symmetric functions $\tilde{\sigma}_1$ and $\tilde{\sigma}_4$ in the $\tilde{x}_i$ have coefficients that are symmetric in the $\beta_i$ and therefore can be expressed in terms of the coefficients of $h$. Thus we see that $\tilde{\sigma}_1=\tilde{\sigma}_4=0$ in $\PP^5$ defines a Burkhardt quartic that over the splitting field of $h$ is isomorphic to $B'$. This provides an easy way of directly constructing twists for Burkhardt quartics. The example $B''$ presented above is also of this type, although the model given makes use of the smaller splitting field.

\end{document}